\documentclass[12pt]{amsart}
\usepackage{amssymb}
\usepackage{bbm}
\usepackage{amsfonts}
\usepackage{latexsym}
\usepackage[all]{xy}
\usepackage{amscd}
\usepackage{comment}
\usepackage{fullpage}
\usepackage{amsmath}
\usepackage{xypic}
\usepackage{textcomp}
\usepackage[pdftex]{graphicx}

\newtheorem{theorem}{Theorem}[section]
\newtheorem{lemma}[theorem]{Lemma}

\theoremstyle{defn}
\newtheorem{defn}[theorem]{Definition}
\newtheorem{prop}[theorem]{Proposition}
\newtheorem{example}[theorem]{Example}

\theoremstyle{remark}

\numberwithin{equation}{section}

\makeatletter

\newcommand{\Rmnum}[1]{\expandafter\@slowromancap\romannumeral #1@}
\makeatother

\begin{document}
\title{FINITE GENERATION OF THE COHOMOLOGY OF QUOTIENTS OF PBW ALGEBRAS}
\date{\today}

\address{Department of Mathematics, Texas A\&M University,
College Station, Texas 77843, USA}
\email{piyushilashroff@gmail.com}

\author{Piyush Shroff} 
\maketitle
\begin{abstract}
In this article we prove finite generation of the cohomology of quotients of a PBW algebra $A$ by relating it to the cohomology of quotients of a quantum symmetric algebra $S$ which is isomorphic to the associated graded algebra of $A$. The proof uses a spectral sequence argument and a finite generation lemma adapted from Friedlander and Suslin.
\end{abstract}
\begin{section}{Introduction}
\indent The cohomology ring of a finite group is finitely generated, as proven by Evens [8], Golod [10] and Venkov [21]. The door to use geometric methods in the study of cohomology and modular representations of finite groups was opened due to this fundamental result. The cohomology ring of any finite group scheme (equivalently, finite dimensional cocommutative Hopf algebra) over a field of positive characteristic is finitely generated, as proven by Friedlander and Suslin [9] which is a  generalization of the result of Venkov and Evens. In [12], Ginzburg and Kumar proved that cohomology of quantum groups at roots of unity is finitely generated. In [7], Etingof and Ostrik conjectured finite generation of cohomology in the context of finite tensor categories. The task of proving this conjecture was done by Mastnak, Pevtsova, Schauenburg and Witherspoon [16] for some classes of noncocommutative Hopf algebras over a field of characteristic~0 .\\
\indent In [16], Mastnak, Pevtsova, Schauenburg and Witherspoon considered the Nichols algebra $R$. A finite filtration on $R$ is used to define a spectral sequence to which they apply a finite generation lemma adapted from [9]. In order to do so, they define 2-cocycles on $R$ that are identified with permanent cycles in the spectral sequence. Finally, they identify the permanent cycles belonging to the degree 2 cohomology of the associated graded algebra of $R$ with elements in the cohomology of $S$ (where $S$ is a quantum symmetric algebra subject to the relation $x_i^{N_i} = 0$ for all $i$) constructed in Section 4 of [16].\\
\indent In this article, we generalize the work done by Mastnak, Pevtsova, Schauenburg and Witherspoon [16] by choosing our parameters that are not necessarily roots of unity and we allow non-nilpotent generators. Also we deal with PBW algebras in general, whereas in [16] authors looked at those that arise from subalgebras of pointed Hopf algebras. Let $k$ be a field, usually assumed to be algebraically closed and of characteristic 0. Let $B$ be a PBW algebra over $k$ generated by $x_1,\cdots,x_{t},\cdots,x_n$ and $A=B/(x_1^{N_1},\cdots,x_{t}^{N_{t}})$ where for each $i,\ 1\leq i\leq t$, $N_i$ is an integer greater than 1 and $x_i^{N_i}$ is in the braided center. Our proof of finite generation of cohomology of the algebra $A$, is a two step procedure. First, we compute cohomology explicitly via a free $S$-resolution where $S$ is a quotient of a quantum symmetric algebra by the ideal generated by $x_1^{N_1},\cdots,x_{t}^{N_{t}}$ where $1\leq t\leq n$. Second, our algebra $A$ has a filtration [5,~Theorem~4.6.5] for which the associated graded algebra $(\text{Gr} A)$ is $S$.\\
\indent This work can be applied to Frobenius-Lusztig kernels studied by Drupieski [6], pointed Hopf algebras studied by Helbig [13] and algebras studied by Liu [15].\\\\
{\bf Notation:} H$^r(A,k)= \text{Ext}_A^{r}(k,k)$ and H$^*(A,k)=\bigoplus_{r\geq 0}\text{H}^r(A,k)$.\\\\
{\bf Main Theorem:} The cohomology algebra H$^{*}(A,k)$ is finitely generated.\\\\
\indent We use the techniques of Mastnak, Pevtsova, Schauenberg and Witherspoon [16] to yield results in this general setting. However, some difference do arise, namely we cannot apply [16, Lemma 2.5] as it is since our parameters are not necessarily roots of unity.\\\\
We now describe the contents of this article:\\\\
\indent In Section 2 we define PBW algebras. In addition, we introduce a result from Evens [8] and a non-commutative version of a finite generation lemma adapted from Friedlander and Suslin [9].\\
\indent In Section 3 we prove that cohomology of quotients of a quantum symmetric algebra $S$ is finitely generated.\\
\indent Section 4 introduces a 2-cocycle on the algebra $A$. In Section 5 we prove that cohomology of the algebra $A$ is finitely generated.\\
\end{section}
\begin{section}{Definitions and Preliminary Results}
\begin{subsection}{PBW Algebras}
In this subsection we recall some basic definitions including that of a PBW algebra.
\begin{defn} 
{\em An admissible ordering on $\mathbb{N}^n$ is a total ordering $<$ such that}\\
{\em 1) if $\alpha < \beta \text{ and } \gamma \in \mathbb{N}^n$ then $\alpha+\gamma < \beta+\gamma$}\\
{\em 2) $<$ is a well ordering}.
\end{defn}
This definition provides one-to-one correspondence between $\mathbb{N}^n$ and monomials in $k[x_1,\cdots,x_n]$. Also, it helps us to compare each monomials to establish their proper relative positions. Some examples of ordering on $n$-tuples include:
\begin{example} 
{\em Let $\alpha=(\alpha_1,\cdots,\alpha_n)$ and $\beta=(\beta_1,\cdots,\beta_n)\in \mathbb{N}^{n}$. The lexicographic order $<_{lex}$ on $\mathbb{N}^n$ is defined by letting $\beta <_{lex}\alpha$ if the first non zero entry of $\alpha-\beta\in \mathbb{Z}^{n}$ is positive}. 
\end{example}
For more examples of ordering on $n$-tuples we refer reader to [5].\\\\
\indent In light of this definition and example we define a PBW algebra.\\\\
{\bf Poincar\'{e}-Birkhoff-Witt Algebra:} A PBW algebra $R$, over a field $k$, is a $k$-algebra together with elements $x_1,\cdots,x_n \in R$ and an admissible order on $\mathbb{N}^n$ for which there are scalars $q_{ij}\in k^{*}$ such that \\
1) $\{x_1^{\alpha_1}\cdots x_n^{\alpha_n}\mid (\alpha_1,\cdots,\alpha_n)\in \mathbb{N}^n\}$ is a basis of $R$ as a $k$-vector space. We call this basis the PBW basis.\\
2) $x_ix_j = q_{ij}x_jx_i + p_{ij}$ for $p_{ij}\in$ $R$, $1\leq i<j\leq n$ where degree of $p_{ij}$ is smaller than that of $x_ix_j$ for the choice of ordering.\\\\
Let us now give some examples of PBW algebras.
\begin{example} 
{\em 1) The polynomial ring $R = k[x_1,x_2,\cdots,x_n]$ is a PBW algebra.\\
2) There are some quantum groups which are PBW algebras. For example:\\ a) The quantum plane $k_q[x,y] = k\langle x,y \mid yx = qxy\rangle$\\ 
b) $U_q(sl_3)^{+}:= k\langle x_1, x_2, x_3 \mid x_{1}x_2 = qx_{2}x_1,\ x_{2}x_3 = qx_{3}x_2,\ x_{1}x_3 = q^{-1}x_{3}x_{1} + x_2\rangle$\\
3) Quantum Symmetric Algebra: Let $k$ be a field. Let $n$ be a positive integer and for each pair $i,j$ of elements in $\{1,\cdots,n\}$, let $q_{ij}$ be a nonzero scalar such that $q_{ii} = 1$ and $q_{ji} = q_{ij}^{-1}$ for all $i,j$. Denote by ${\bf q}$ the corresponding tuple of scalars, ${\bf q} := (q_{ij})_{1\leq i<j\leq n}.$ Let $V$ be a vector space with basis $x_1,\cdots,x_n$, and let
$$ S_{\bf q}(V) := k\langle x_1,\ldots,x_n \mid x_ix_j = q_{ij}x_jx_i  \mbox{ for all }
      1\leq i<j\leq n \rangle,$$
the} quantum symmetric algebra {\em(quantum polynomial ring)} {\em determined by $\bf q$}.
\end{example}
\noindent {\bf The $\omega$-filtration of a PBW algebra:}\\
Let $\omega = (\omega_1,\cdots,\omega_n)\in \mathbb{N}^n$. For $0\neq f$ belonging to a PBW algebra $R$ we define its $\omega$-degree as
$$\text{deg}_\omega(f) = \text{max}\{|\alpha|_\omega\mid \alpha\in \mathcal{W}\}$$
where $|\alpha|_\omega =\alpha_1\omega_1+\cdots+\alpha_n\omega_n$, $f=\sum_{\alpha\in \mathbb{N}^n}c_{\alpha}x^{\alpha}$ and $\mathcal{W} =\{\alpha\in \mathbb{N}^n\mid c_\alpha\neq 0\}$. With these notations we define the $\omega$-filtration of a PBW algebra as
$$ F_s^{\omega}R=\{f\in R\mid |\alpha|_\omega\leq s \text{ for all } \alpha\in \mathcal{W}\}$$
where $s$ is any nonnegative integer (See [5]).\\
\end{subsection}
\begin{subsection}{Noetherian Modules}
Given a ring $R$, a decreasing filtration $F^{n}R$ for $n\in\mathbb{N}$ is called compatible with the ring structure on $R$ if $\ F^{m}R\cdot F^{n}R\subset F^{m+n}R$, for all $m, n\in \mathbb{N}$. The ring $R$ with this filtration is then called a filtered ring (See [4]). Let $R=F^{0}R\supseteq F^{1}R\supseteq \cdots \supseteq F^{s}R\supseteq \cdots$ be a graded filtered ring. Note that by definition, the grading on $R$ is compatible with its ring structure in the usual way that is $R=\bigoplus_{n\in\mathbb{N}}R^n$, and $R^nR^m\subset R^{n+m}$. Then we may form the doubly graded ring $$E_0(R) = \sum_i F^{i}R/F^{i+1}R.$$ Similarly we may form the doubly graded module $E_0(N)$ over $E_0(R)$ if $N$ is a graded filtered module over $R$ (with the module structure consistent with  the ring structure in the usual way that is $N=\bigoplus_{i\in \mathbb{N}}N^i$, and $R^iN^j\subset N^{i+j}$).\\
\indent For the current purposes it is sufficient to consider filtrations such that $F^{i}R^n = 0$ for $i$ sufficiently large where $n$ denotes the grading on $R$. Similarly, $F^{i}N^j = 0$ for $i$ sufficiently large.\\\\
\indent Now we define a couple of terms and recall the following proposition of Evens [8].
\begin{defn} {\em 1) A submodule $S\subset N$ is said to be homogeneous if it is generated by homogeneous elements (i.e. the elements from homogeneous summands $N^i$).\\\\
2) An $R$-submodule $N$ of a graded $R$-module $M$ is called a graded $R$-submodule of $M$ if we have $N=\bigoplus_s{(N\cap M^s)}.$\\\\
3) If $\{F^{s}M\}$ is a filtration of the $R$-module $M$, and $N$ is a submodule of $M$, then we have a filtration induced on $N$, given by $F^{s}N = N\cap F^{s}M$}.\\
\end{defn}
\begin{prop}\label{Pch2}
Let $R$ be a graded filtered ring i.e. $$R=F^{0}R\supseteq F^{1}R\supseteq \cdots \supseteq F^{s}R\supseteq \cdots$$ and $N$ a graded filtered $R$ module i.e. suppose $$N=F^{0}N\supseteq F^{1}N\supseteq \cdots \supseteq F^{s}N\supseteq \cdots$$ over $R$. If $E_0(N)$ is (left) Noetherian over $E_0(R)$, then $N$ is Noetherian over $R$.
\end{prop}
\begin{proof}
See [8, Section 2, Proposition 2.1] and [20, Chapter 2].
\end{proof}
\vspace*{.5cm}
\noindent{\bf A finite generation lemma.} In Section 5, we will need the following general lemma which is a non-commutative version of [16, Lemma 2.5] and is originally adapted from [9,~Lemma~1.6]. Recall that an element $x \in E_r^{p,q}$ is called a $permanent \ cycle$ if $d_i(x)=0$ for all $i\geq r.$ More precisely, if $i> r,\ d_i$ is applied to the image of $x$ in $E_i$.\\
\begin{lemma}\label{Lch2}
a) Let $E_1^{p,q} \Rightarrow E_{\infty}^{p+q}$ be a multiplicative spectral sequence of bigraded $k$-algebras concentrated in the half plane $p+q \geq 0$ and let $C^{*,*}$ be a bigraded $k$-algebra. For each fixed $q$, assume that $C^{p,q} = 0$ for $p$ sufficiently large. Assume that there exists a bigraded map of algebras $\phi: C^{*,*}\rightarrow E_1^{*,*}$ such that \\
1) $\phi$ makes $E_1^{*,*}$ into a left Noetherian $C^{*,*}$-module, and\\
2) the image of $C^{*,*}$ in $E_1^{*,*}$ consists of permanent cycles.\\
Then $E_\infty^{*}$ is a left Noetherian module over Tot($C^{*,*}$).\\
b) Let $\widetilde{E}_1^{p,q} \Rightarrow \widetilde{E}_\infty^{p+q}$ be a spectral sequence that is a bigraded module over the spectral sequence $E^{*,*}$. Assume that $\widetilde{E}_1^{*,*}$ is a left Noetherian module over $C^{*,*}$ where $C^{*,*}$ acts on $\widetilde{E}_1^{*,*}$ via the map $\phi$. Then $\widetilde{E}_\infty^{*}$ is a finitely generated $E_\infty^{*}$-module.
\end{lemma}
\begin{proof}
Let $\Lambda_r^{*,*}\subset E_r^{*,*}$ be the bigraded subalgebra of permanent cycles in $E_r^{*,*}.$\\
We claim first that $d_r(E_r^{*,*})\subset \Lambda_r^{*,*}$. In order to see this note that $d_r(E_r^{*,*})=$ im($d_r$). Therefore, $d_r(E_r^{*,*})\subset$ Ker\ $d_{r+1}$. Hence, $d_{r+1}(d_r(E_r^{*,*}))=0$. Similarly, $d_{r+2}(d_r(E_r^{*,*}))=0$ and so on. Thus, we have $d_i(d_r(E_r^{*,*})) = 0$ for all $i\geq r$. Hence, $d_r(E_r^{*,*})\subset \Lambda_r^{*,*}.$\\
\indent Next we claim that for all $\lambda \in \Lambda_r^{*,*}$ and $\mu\in E_r^{*,*}$, $\lambda\cdot d_r(\mu) \in d_r(E_r^{*,*})$ that is, $d_r(E_r^{*,*})$ is a left ideal of $\Lambda_r^{*,*}$. Consider,
\begin{eqnarray*} 
&&\hspace{-2.27cm}d_r(\lambda\cdot \mu)= d_r(\lambda)\mu + (-1)^{p+q}\lambda\cdot d_r(\mu)\ \text{ where } \lambda \in \Lambda^{p,q}\\
&=&0 + (-1)^{p+q}\lambda\cdot d_r(\mu)
\end{eqnarray*}
So $\lambda\cdot d_r(\mu) \in d_r(E_r^{*,*}).$ Thus $d_r(E_r^{*,*})$ is a left ideal of $\Lambda_r^{*,*}.$\\
\indent Now the image of $C^{*,*}$ is contained in each page of the spectral sequence and by assumption it consists of permanent cycles. Hence, we can  similarly conclude as above that $d_r(E_r^{*,*})$ is a $C^{*,*}$-submodule.\\
\indent A similar computation as above shows that $\Lambda_1^{*,*}$ is a $C^{*,*}$-submodule of $E_1^{*,*}$. To see this let $a\in C^{p,q}$; therefore, $\phi(a) \in E_1^{*,*}$ and $\lambda_1\in \Lambda_1^{*,*}$. Consider,
\begin{eqnarray*}
&&\hspace{-2.68cm}d_i(\phi(a)\lambda_1)= d_i(\phi(a))\lambda_1 + (-1)^{p+q}\phi(a)d_i(\lambda_1)\ \text{ where } i\geq 1\\ 
&=&0+0 = 0
\end{eqnarray*}
So $\phi(a)\lambda_1\in \Lambda_1^{*,*}$. Thus $\Lambda_1^{*,*}$ is an $C^{*,*}$-submodule.\\ 
\indent By induction, $\Lambda_{r+1}^{*,*} = \Lambda_r^{*,*}/d_r(E_r^{*,*})$ is an $C^{*,*}$-module for any $r\geq 1$ because $d_r(E_r^{*,*})\subset \Lambda_r^{*,*}$ and by the induction hypothesis $\Lambda_r^{*,*}$ is a $C^{*,*}$-module. Therefore,\\ $\Lambda_r^{*,*}/d_r(E_r^{*,*})$ is a $C^{*,*}$-module that is, $\Lambda_{r+1}^{*,*}$ is a $C^{*,*}$-module.\\
\indent We get a sequence of surjective maps of $C^{*,*}$-modules:
\begin{equation}
\Lambda_1^{*,*}\twoheadrightarrow\Lambda_2^{*,*}\twoheadrightarrow\cdots \twoheadrightarrow\Lambda_r^{*,*}\twoheadrightarrow\Lambda_{r+1}^{*,*}\twoheadrightarrow\cdots\\
\end{equation}
Since $\Lambda_1^{*,*}$ is a $C^{*,*}$-submodule of $E_1^{*,*}$, it is Noetherian as a $C^{*,*}$-module. Therefore, the kernels of the maps $\Lambda_1^{*,*}\twoheadrightarrow\Lambda_r^{*,*}$ are Noetherian for all $r\geq 1$. These kernels form an increasing chain of submodules of $\Lambda_1^{*,*}$; hence, by the Noetherian property, they stabilize after finitely many steps; that is, $\Lambda_r^{*,*} = \Lambda_{r+1}^{*,*} = \cdots$ for some $r$. We conclude that $\Lambda_r^{*,*} = E_\infty^{*,*}$. Therefore $E_\infty^{*,*}$ is a Noetherian $C^{*,*}$-module. Also, both $E_\infty^{*,*}$ and $C^{*,*}$ are filtered algebras and the filtration for each $n$ is given by:
$$E_{\infty}^n = \bigoplus_{p+q=n}E_{\infty}^{p,q}\supseteq \bigoplus_{\substack{p+q=n\\p\geq 1}}E_{\infty}^{p,q}\supseteq \bigoplus_{\substack{p+q=n\\ p\geq 2}}E_{\infty}^{p,q}\supseteq\cdots$$
and $E_\infty^{*,*}$ is the associated graded algebra. Similarly, for each $n$:
$$C^n=\bigoplus_{p+q=n}C^{p,q}\supseteq \bigoplus_{\substack{p+q=n\\p\geq 1}}C^{p,q} \bigoplus_{\substack{p+q=n\\ p\geq 2}}C^{p,q}\supseteq\cdots$$
and $C^{*,*}$ is the associated graded algebra.\\
\indent For $p$ sufficiently large, $C^{p,q} = 0$. Hence, by proposition \ref{Pch2}, $E_\infty^{*}$ is a Noetherian module over Tot($C^{*,*}$).\\\\
(b) Similarly, we can show that $\widetilde{E}_\infty^{*,*}$ is Noetherian over $C^{*,*}$. Again, by applying Proposition~\ref{Pch2}, we can conclude that $\widetilde{E}_\infty^{*}$ is Noetherian and hence finitely generated over Tot($C^{*,*}$). Therefore, by part (a) $\widetilde{E}_\infty^{*}$ is a Noetherian module over $E_\infty^{*}$. Hence, $\widetilde{E}_\infty^{*}$ is finitely generated over $E_\infty^{*}$.\\
\end{proof}
\end{subsection}
\end{section}
\begin{section}{Cohomology of Quotients of Quantum Symmetric Algebras}
\indent For this section we will use the same terminology as used by Mastnak, Pevtsova, Schauenburg and Witherspoon in Section 4 of [16]. We will make some modifications to their method to accommodate non-nilpotent generators. This will enable us to generalize their method.\\\\
\indent Let $n,\ t$ with $t\leq n$ be positive integers, and for each $i,\ 1\leq i\leq t$, let $1< N_i\in \mathbb{Z}$. Let $q_{ij}\in k^{*}$ for $1\leq i<j\leq n$ with $q_{ji}=q_{ij}^{-1}$ for $i<j$ and $q_{ii}=1$. Let
\begin{equation}\label{Ech3}
S=k\langle x_1,\cdots,x_{t},\cdots,x_n\mid x_ix_j = q_{ij}x_jx_i \text{ for all } i < j\text{ and } x_i^{N_i} = 0 \text{ for } 1\leq i\leq t \rangle.
\end{equation}
{\bf Note:} If $t=0$, in the special case coming from small quantum groups, Ginzburg and Kumar [12] show that cohomology of $S$ is a quantum exterior algebra so is finitely generated. The same should follow for more general $S$, as a direct calculation using explicit resolution.\\\\
\indent We will compute H$^{*}(S,k)=$ Ext$_{S}^{*}(k,k)$ where $k$ is an $S$-module on which $x$ acts as multiplication by zero. Bergh and Oppermann [3, Lemma 3.6] tells us that twisted tensor product of two resolutions is again a resolution. We sketch the proof for completeness as well as the construction is needed for later sections. To obtain information at the chain level, we need an explicit free $S$-resolution of $k$. This resolution is originally adapted from [2] and it is a twisted tensor product of the periodic resolutions
$$
\cdots\stackrel{x_i^{N_i-1}\cdot}\longrightarrow k[x_i]/(x_i^{N_i})\stackrel{x_i\cdot}\longrightarrow k[x_i]/(x_i^{N_i}) \stackrel{x_i^{N_i-1}\cdot}\longrightarrow k[x_i]/(x_i^{N_i})\stackrel{x_i\cdot}\longrightarrow k[x_i]/(x_i^{N_i})\stackrel{\varepsilon}\rightarrow k\rightarrow 0,$$
one for each $i,\ 1\leq i\leq t$ and
$$0\rightarrow k[x_i] \stackrel{x_i\cdot}\longrightarrow k[x_i]\stackrel{\varepsilon}\rightarrow k\rightarrow 0,$$
one for each $i,\ t+1\leq i\leq n$. 
 
Explicitly, we define a complex $K_\bullet$ of free $S$-modules as follows. For each $n$-tuple $(a_1,\cdots,a_n)$ of non-negative integers with $a_i = 0 \text{ or } 1$\, for each $i$,\ $t+1\leq i\leq n$, let $\Phi(a_1,\cdots,a_n)$ be a free generator in degree $a_1+\cdots+a_n$. Thus $$K_m = \oplus_{a_1+\cdots+a_n=m}S\Phi(a_1,\cdots,a_n).$$
{\bf Note:} Throughout this section we will interpret $\Phi(a_1,\cdots,a_i-1,\cdots,a_n)=0$ if $a_i-1$ is negative. Similarly, $\Phi(a_1,\cdots,a_i-2,\cdots,a_n)$ and $\Phi(a_1,\cdots,a_i-3,\cdots,a_n)$ will be zero if $a_i-2$ and $a_i-3$ are negative respectively.\\\\
For each $i,\ 1\leq i \leq t,$ let $\sigma_i, \tau_i : \mathbb{N}\rightarrow \mathbb{N}$ be the functions defined by
$$\sigma_i(a) =
\begin{cases}
1, \text{ if } a \text{ is odd}\\ 
N_i-1, \text{ if } a \text{ is even},
\end{cases}
$$
and $\tau_i(a) = \sum_{j=1}^a \sigma_i(j) \text{ for } a\geq 1, \tau_i(0) = 0.$ For each $ i,\ t+1\leq i\leq n$ we define $\sigma_i(a) = 1$ and $\tau_i(a) = a.$\\\\
We define the differential as follows:
$$d_i(\Phi(a_1,\cdots,a_n))=
\begin{cases}
\prod_{i<l} (-1)^{a_l} q_{li}^{\sigma_i(a_i)\tau_l(a_l)}x_i^{\sigma_i(a_i)} \Phi(a_1,\cdots,a_i-1,\cdots,a_n),\ \text{ if } a_i>0\\
\hspace{1.55in}0,\ \hspace{2in} \text{ if } a_i=0
\end{cases}
$$
Extend each $d_i$ to an $S$-module homomorphism. We will now verify that $K_{\bullet}$ is a complex. Let $d=d_1+\cdots+d_n$. Note that $x_i^{N_i} =0$ when $i\leq t$ and $\sigma_i(a_i)+\sigma_i(a_i-1)=N_i$. Consider,
\begin{eqnarray*}
d_id_i(\Phi(a_1,\cdots,a_n))\\
&=&d_i\left(\left(\prod_{i<l} (-1)^{a_l} q_{li}^{\sigma_i(a_i)\tau_l(a_l)}\right)x_i^{\sigma_i(a_i)} \Phi(a_1,\cdots,a_i-1,\cdots,a_n)\right)\\
&=&\left(\prod_{i<l} (-1)^{a_l} q_{li}^{\sigma_i(a_i)\tau_l(a_l)}\right)\left(\prod_{i<m} (-1)^{a_m} q_{mi}^{\sigma_i(a_i-1)\tau_m(a_m)}\right)\\
&&\hspace{1in}\cdot x_i^{\sigma_i(a_i)}x_i^{\sigma_i(a_i-1)} \Phi(a_1,\cdots,a_i-2,\cdots,a_n)\\
&=&0
\end{eqnarray*}
Since for $i>t,\ a_i-2$ is negative and in fact we get $0$ by definition of $d_i$. If $i\leq t$, it is because $x_i^{N_i}=0$.\\
If $i<j$, we have
\begin{eqnarray*}
d_id_j(\Phi(a_1,\cdots,a_n))\\
&=&d_i\left(\left(\prod_{j<l} (-1)^{a_l} q_{lj}^{\sigma_j(a_j)\tau_l(a_l)}\right)x_j^{\sigma_j(a_j)} \Phi(a_1,\cdots,a_j-1,\cdots,a_n)\right)\\
&=&\left(\prod_{j<l} (-1)^{a_l} q_{lj}^{\sigma_j(a_j)\tau_l(a_l)}\right)\left(\prod_{i<m} (-1)^{a_m} q_{mi}^{\sigma_i(a_i)\tau_m(a_m)}\right)\\
&&\hspace{.5in}\cdot x_j^{\sigma_j(a_j)}x_i^{\sigma_i(a_i)}\Phi(a_1,\cdots,a_i-1,\cdots,a_j-1,\cdots,a_n).
\end{eqnarray*}
\begin{eqnarray*}
d_jd_i(\Phi(a_1,\cdots,a_n))\\
&=&d_j\left(\left(\prod_{i<m} (-1)^{a_m} q_{mi}^{\sigma_i(a_i)\tau_m(a_m)}\right)x_i^{\sigma_i(a_i)} \Phi(a_1,\cdots,a_i-1,\cdots,a_n)\right)\\
&=&\left(\prod_{i<m} (-1)^{a_m} q_{mi}^{\sigma_i(a_i)\tau_m(a_m)}\right)\left(\prod_{j<l} (-1)^{a_l} q_{lj}^{\sigma_j(a_j)\tau_l(a_l)}\right)\\
&&\hspace{.5in}\cdot x_i^{\sigma_i(a_i)}x_j^{\sigma_j(a_j)}\Phi(a_1,\cdots,a_i-1,\cdots,a_j-1,\cdots,a_n).
\end{eqnarray*}
Comparison shows that a scalar factor for the term in which $m=j$ changes from $(-1)^{a_j} q_{ji}^{\sigma_i(a_i)\tau_j(a_j)}$ to $(-1)^{a_j-1} q_{ji}^{\sigma_i(a_i)\tau_j(a_j-1)}$, and $x_j^{\sigma_j(a_j)}x_i^{\sigma_i(a_i)}$ is replaced by $x_j^{\sigma_j(a_j)}x_i^{\sigma_i(a_i)}=q_{ji}^{\sigma_i(a_i)\sigma_j(a_j)}x_i^{\sigma_i(a_i)}x_j^{\sigma_j(a_j)}$. Since $\tau_i(a_i)=\tau_i(a_i-1)+\sigma_i(a_i)$, this illustrates that
$$d_id_j+d_jd_i = 0.$$
Since $d^2 = 0$, we can say that $K_\bullet$ is indeed a complex.\\
\indent Next we give a contracting homotopy to show that $K_\bullet$ is a resolution of $k$:\\
Let $\eta \in S$, and fix $l, 1\leq l\leq n$. Write 
$$\eta = 
\begin{cases}
\sum_{j=0}^{N_i-1} \eta_j x_l^{j}, \text{ if } 1\leq l\leq t \\\\
\sum_j \eta_j x_l^{j}, \text{ if } t+1\leq l\leq n
\end{cases}
$$
where $\eta_j$ is in the subalgebra of $S$ generated by the $x_i$ with $i \neq l$. Define\\
$s_l(\eta\Phi(a_1,\cdots,a_n))$
$$=
\begin{cases}
\displaystyle\sum_{j=0}^{N_i-1}s_l(\eta_jx_l^{j}\Phi(a_1,\cdots,a_n)), \text{ if } 1\leq l\leq t \\\\
\displaystyle\sum_j s_l(\eta_jx_l^{j}\Phi(a_1,\cdots,a_n)), \text{ if } t+1\leq l\leq n
\end{cases}
$$
where \\
$s_l(\eta_jx_l^{j}\Phi(a_1,\cdots,a_n))$
$$=
\begin{cases}
\delta_{j>0}(\prod_{l<m} (-1)^{a_m}q_{ml}^{-\sigma_l(a_l+1)\tau_m(a_m)})\eta_jx_l^{j-1}\Phi(a_1,\cdots,a_l+1,\cdots,a_n),\\
\hspace{2.5in} \text{ if } a_l \text{ is even with } 1\leq l\leq t\\
\delta_{j,N_l-1}(\prod_{l<m} (-1)^{a_m}q_{ml}^{-\sigma_l(a_l+1)\tau_m(a_m)})\eta_j\Phi(a_1,\cdots,a_l+1,\cdots,a_n),\\
\hspace{2.5in} \text{ if } a_l \text{ is odd with } 1\leq l\leq t\\
\omega \eta_jx_l^{j-1}\Phi(a_1,\cdots,a_{l}+1,\cdots,a_n),\ \text{ if } t+1\leq l\leq n
\end{cases}
$$\\\\
where \hspace{.1cm}$\delta_{j>0} =
\begin{cases}
1, \text{ if } j>0\\ 
0, \text{ if } j=0
\end{cases}
$\hspace{.1cm} and \hspace{.1cm} $\omega = \dfrac{1}{\prod_{l<u}(-1)^{a_u}q_{ul}^{a_u}}.$\\\\
\indent With the help of calculations we see that for all $i,\ 1\leq i\leq n$, 
$$(s_id_i+d_is_i)(\eta_jx_i^{j}\Phi(a_1,\cdots,a_n)) =
\begin{cases}
\eta_jx_i^{j}\Phi(a_1,\cdots,a_n), \text{ if } j>0 \text{ or } a_i>0\\
\hspace{1.5cm}0,\hspace{1.42cm} \text{ if } j=0 \text{ and } a_i =0
\end{cases}
$$\\
The way we have defined our $s_l$ and $d_i$, we get $s_ld_i+d_is_l = 0$ for all $i,l$ when $i\neq l$. For each $x_1^{j_1}\cdots x_n^{j_n}\Phi(a_1,\cdots,a_n)$, let $C = c_{j_1,\cdots,j_n,a_1,\cdots,a_n}$ be the cardinality of the set of all $i (1\leq i\leq n)$ such that both $j_i$ and $a_i$ are 0. Define\\
$$s(x_1^{j_1}\cdots x_n^{j_n}\Phi(a_1,\cdots,a_n)) = 
\begin{cases}
\dfrac{1}{n -C}(s_1+\cdots+s_n)(x_1^{j_1}\cdots x_n^{j_n}\Phi(a_1,\cdots,a_n))\\\\
0,\ \text{ if } n=C
\end{cases}
$$\\
and since $d=d_1+\cdots+d_n$, we have $sd+ds = id$ on each $K_m, m>0$. That is, $K_\bullet$ is exact in positive degrees. For exactness at $K_0 = S$ we look at the kernel of the augmentation (counit) map $\varepsilon : S\rightarrow k$ and the image of $d(x_i^{j_i-1}\cdots x_n^{j_n}\Phi(0,\cdots,1,\cdots,0))$. Observe that the kernel of $\varepsilon$ is spanned over the field $k$ by the elements $x_1^{j_1}\cdots x_t^{j_t}x_{t+1}^{j_{t+1}}\cdots x_n^{j_n} 
\cdot\Phi(0,\cdots,0),\ 0\leq j_i\leq N_i$, for $1\leq i\leq t$ and $j_i\in {\mathbb{N}}$ for $t +1\leq i\leq n$, with at least one $j_i\neq 0$.  Assume that $x_1^{j_1}\cdots x_t^{j_t}x_{t+1}^{j_{t+1}}\cdots x_n^{j_n}\Phi(0,\cdots,0)$ is such an element, and assume that $i$ is the smallest positive integer such that $j_i\neq 0$. Then $$d(x_i^{j_i-1}\cdots x_n^{j_n}\Phi(0,\cdots,1,\cdots,0))=\vartheta x_i^{j_i}\cdots x_t^{j_t}x_{t+1}^{j_{t+1}}\cdots x_n^{j_n}\Phi(0,\cdots,0)$$ where $\vartheta$ is a nonzero scalar. Thus ker$(\varepsilon)$ = im$(d),$ and $K_\bullet$ is a free resolution of $k$ as an $S$-module.\\
\indent Next to compute Ext$_S^{*}(k,k)$ we apply Hom$_S(-,k)$ to $K_\bullet$. The Hom$_S(-,k)$ functor induces the differential $d^{*}$. Moreover it is the zero map since $x_i$'s act as 0 on $k$. Thus the cohomology is the complex Hom$_S(K_\bullet,k)$. Now let  $\eta_i\in$ Hom$_S(K_1,k)$ be the function dual to $\Phi(0,\cdots,0,1,0\cdots,0)$ (the 1 in the $i$th position) and $\xi_i\in$ Hom$_S(K_2,k)$ for $i\leq t$ be the function dual to $\Phi(0,\cdots,0,2,0\cdots,0)$ (the 2 in the $i$th position). By abusing the notation we will identify the functions $\xi_i, \eta_i$ with the corresponding elements in H$^2(S,k)$ and H$^1(S,k)$, respectively. Further, we will show that they generate H$^{*}(S,k)$ and determine the relations among them. To do this we denote by $\xi_i$ and $\eta_i$ the corresponding chain maps $\xi_i: K_{\bullet}\rightarrow K_{\bullet -2}$ and $\eta_i: K_{\bullet}\rightarrow K_{\bullet -1}$ defined by
\begin{align*}
\xi_i(\Phi(a_1,\cdots, a_n))&=\displaystyle\prod_{l<i}q_{il}^{N_i\tau_l(a_l)}\Phi(a_1,\cdots, a_i-2,\cdots, a_n),\ \text{ if } 1\leq i\leq t \\\\
\eta_i(\Phi(a_1,\cdots, a_n))&=\prod_{i<l}q_{li}^{(\sigma_i(a_i)-1)\tau_l(a_l)}\prod_{l<i}(-1)^{a_l}q_{il}^{\tau_l(a_l)}x_i^{\sigma_i(a_i)-1}\Phi(a_1,\cdots, a_i-1,\cdots, a_n)\\
\end{align*}
\begin{theorem}\label{Tch3}
Let $S$ be the $k$-algebra generated by $x_1,\cdots,x_n$, subject to relations\\ $x_ix_j = q_{ij}x_jx_i\ \text{ for all }\ i < j,\ x_i^{N_i}= 0\ \text{ for }\ 1\leq i \leq t.$ Then {\em H}$^{*}(S,k)$ is generated by $\xi_i\ (i=1,\cdots,t)$ and $\eta_i\ (i=1,\cdots,n)$ where deg $\xi_i=2$ and deg $\eta_i=1$, subject to the relations
$$\xi_i\xi_j=q_{ji}^{N_iN_j}\xi_j\xi_i,\ \eta_i\xi_j=q_{ji}^{N_j}\xi_j\eta_i,\ \text{ and }\ \eta_i\eta_j=-q_{ji}\eta_j\eta_i.$$
\end{theorem}
\begin{proof}
The ring structure of the subalgebra of H$^{*}(S,k)$ generated by $\xi_i, \eta_i$ is given by composition of these chain maps.\\
\indent A calculation shows that the relations hold and if $N_i=2$, then $\eta_i^2$ is a nonzero scalar multiple of $\xi_i$ and the corresponding element in cohomology is zero if $N_i\neq 2$. Thus any element in the algebra generated by the $\xi_i$ and $\eta_i$ may be written as a linear combination of elements of the form 
$\xi_1^{b_1}\cdots\xi_{t}^{b_t}\eta_1^{c_1}\cdots\eta_{t}^{c_t}\cdots\eta_n^{c_n}$ with $b_i\geq 0$ and $c_i\in \{0,1\}$.\\ 
\indent We claim that the set of all $\xi_1^{b_1}\cdots\xi_{t}^{b_t}\eta_1^{c_1}\cdots\eta_{t}^{c_t}\cdots\eta_n^{c_n}$ forms a $k$-basis for H$^{*}(S,k)$.\\
First, calculation shows that
\begin{align*}
\xi_1^{b_1}\cdots\xi_{t}^{b_t}\eta_1^{c_1}\cdots\eta_{t}^{c_t}\cdots\eta_n^{c_n}(\Phi(2b_1+c_1,\cdots,2b_{t}+c_{t},c_{t+1},\cdots,c_n))\\
&\hspace{-3in}=\nu\Phi(0,\cdots,0)
\end{align*}
where $\nu$ is some nonzero scalar and
\begin{align*}
\xi_1^{b_1}\cdots\xi_{t}^{b_t}\eta_1^{c_1}\cdots\eta_{t}^{c_t}\cdots\eta_n^{c_n}(\Phi(e_1,\cdots,e_{t},e_{t+1},\cdots,e_n))=0
\end{align*}
where $e_i\neq 2b_i+c_i$ for some $i$. That is, $\xi_1^{b_1}\cdots\xi_{t}^{b_t}\eta_1^{c_1}\cdots\eta_{t}^{c_t}\cdots\eta_n^{c_n}$ takes all other $S$-basis elements of $K_{\sum (2b_i+c_i)}$ to 0. Therefore, all such monomials form a linearly independent set.\\
\indent Clearly in each degree, there are the same number of elements of the form $\xi_1^{b_1}\cdots\xi_{t}^{b_t}\eta_1^{c_1}\cdots\eta_{t}^{c_t}\cdots\eta_n^{c_n}$ as there are free generators $\Phi(a_1,\cdots,a_n)$. Therefore, the $\xi_1^{b_1}\cdots\xi_{t}^{b_t}\eta_1^{c_1}\cdots\eta_{t}^{c_t}\cdots\eta_n^{c_n}$ must form a dual basis to the $\Phi(a_1,\cdots,a_n)$.\\
\indent Hence, we get that H$^{*}(S,k)\cong$ Hom$_S(K_\bullet,k) \cong$ Hom$_k(V,k)$ where $V$ has basis all $\Phi(a_1,\cdots,a_n)$. This shows that the set of monomials of the form $\xi_1^{b_1}\cdots\xi_{t}^{b_t}\cdots\xi_n^{b_n}\eta_1^{c_1}\cdots\eta_{t}^{c_t}\cdots\eta_n^{c_n}$ forms a $k$-basis for H$^{*}(S,k)$.\\
\end{proof}
\end{section}
\begin{section}{Some Cocycles on The Algebra}
For this section we will use the same terminology as used by Mastnak and Witherspoon in Section 6 of [17] with some additional information.\\\\
\indent Let $B$ be a PBW algebra over $k$ as defined in Section 2 and $A=B/(x_1^{N_1},\cdots,x_t^{N_t})$. As a vector space $B$ has a basis $\{x_1^{i_1}x_2^{i_2}\cdots x_n^{i_n}\mid i_1,\cdots,i_n\in\mathbb{N}\}.$\\
\indent We want to show that the above set is indeed a basis for $A$ with some restriction on $i_j, 1\leq j\leq t$. To prove its a basis we need the assumption that $x_i^{N_i}$ is in the braided center of $B$ for all $i,\ 1\leq i\leq t$.\\
\indent Let $b\in B$. Then $b= \sum_I a_Ix_1^{i_1}x_2^{i_2}\cdots x_n^{i_n}$ is a finite sum where $I=(i_1,i_2,\cdots,i_n)$ and $a_I$ is a scalar. Therefore, 
\begin{align*}
b+(x_1^{N_1},\cdots,x_t^{N_t})&= \sum_Ia_Ix_1^{i_1}x_2^{i_2}\cdots x_n^{i_n}+ (x_1^{N_1},\cdots,x_t^{N_t})\\
&=\sum_{\substack{I\\ 0\leq i_j< N_j\\ 1\leq j\leq t}}(a_Ix_1^{i_1}x_2^{i_2}\cdots x_n^{i_n}+ (x_1^{N_1},\cdots,x_t^{N_t}))
\end{align*}\\
This proves that $\{x_1^{i_1}x_2^{i_2}\cdots x_n^{i_n}\mid 0\leq i_1 <N_1,\cdots,0\leq i_t <N_t,\hspace{0.2cm} i_{t+1},\cdots,i_n\in~\mathbb{N}\}$ is a spanning set for $A$.\\\\
\indent Define, $$[x_1^{i_1}\cdots x_n^{i_n},x_1^{j_1}\cdots x_n^{j_n}]_c = x_1^{i_1}\cdots x_n^{i_n}x_1^{j_1}\cdots x_n^{j_n}- (\prod_{k<l}q_{lk}^{-(j_li_k-j_ki_l)})x_1^{j_1}\cdots x_n^{j_n}x_1^{i_1}\cdots x_n^{i_n}.$$
\begin{defn}
{\em An element of the form $x_1^{i_1}\cdots x_n^{i_n}$ is said to be in the braided center of $B$, if}
\begin{equation}
[x_1^{i_1}\cdots x_n^{i_n},x_1^{j_1}\cdots x_n^{j_n}]_c = 0,\hspace{0.2cm} \text{\em{for all} } x_1^{j_1}\cdots x_n^{j_n}\in~B.\\
\end{equation}
\end{defn}
\indent Assume that $x_i^{N_i}$ is in the braided center of $B$ for all $i,\ 1\leq i\leq t$. This assumption will be also needed for a later part of this section.\\
\indent To show that the set is linearly independent we need to prove that $\displaystyle\sum_Ia_Ix_1^{i_1}x_2^{i_2}\cdots x_n^{i_n}$ belonging to $(x_1^{N_1},\cdots,x_t^{N_t})$ implies all $a_I = 0$.\\
\indent Consider, $$\displaystyle\sum_Ia_Ix_1^{i_1}x_2^{i_2}\cdots x_n^{i_n} =\displaystyle \sum_{J,i} T_Jx_i^{N_i}W_J$$ where $T_J, W_J\in B$. Since $x_i^{N_i}$ is in the braided center we have
$$\displaystyle\sum_Ia_Ix_1^{i_1}x_2^{i_2}\cdots x_n^{i_n} =\displaystyle \sum_{J,i}x_i^{N_i}U_J$$ where $U_J\in B$. Observe that in each expression on the right hand side there is atleast one $i$ for which the  power of $x_i$ is atleast $N_i$. Thus by comparing the coefficients we get $a_I = 0$.\\
\indent Hence, $\{x_1^{i_1}x_2^{i_2}\cdots x_n^{i_n}\mid 0\leq i_1 <N_1,\cdots,0\leq i_t <N_t,\hspace{0.2cm} i_{t+1},\cdots,i_n\in\mathbb{N}\}$ is a basis for $A$.\\\\
\indent Next we want to define 2-cocycles $\zeta_i$ on $A$. These 2-cocycles represent the elements of H$^{2}(A,k)$. We make use of the reduced bar resolution of $k$,
$$\cdots\longrightarrow  B\otimes (B^{+})^{\otimes 2}\stackrel{\delta_2}\longrightarrow B\otimes B^{+}\stackrel{\delta_1}\longrightarrow B\stackrel{\varepsilon}\longrightarrow k\longrightarrow 0.$$
where $B$ is an augmented algebra with augmention map $\varepsilon : B\rightarrow k$, $B^{+} =$ Ker $\varepsilon$ is the augmentation ideal and $\delta_i(b_0\otimes b_1\otimes\cdots\otimes b_i) = \sum_{j=0}^{i-1}(-1)^jb_0\otimes\cdots\otimes b_jb_{j+1}\otimes\cdots\otimes b_i$. For each $i, 1\leq i\leq t$ define $\tilde{\zeta_i} : B^{+}\otimes B^{+}\rightarrow k$ by 
$$\tilde{\zeta_i}(r\otimes s) = \gamma_{(0,\cdots,0,N_i,0,\cdots,0)}$$ 
where $N_i$ is in the $i^{th}$ position and $rs = \sum_a\gamma_ax^a \in B.$ We need to check that $\tilde{\zeta_i}(r\otimes s)$ is associative that is to show that $\tilde{\zeta_i}(rr_1\otimes s) = \tilde{\zeta_i}(r\otimes~r_1s) \text{ for all }\hspace{0.1cm} r, r_1, s\in B^{+}.$
But this is true by definition and thus $\tilde\zeta_i$ may be trivially extended to a 2-cocycle on $B$. Let us see how it is done. We will denote the 2-cocycle on $B$ by $\tilde\xi_i$ and define as  $\tilde\zeta_i(b_1\otimes b_2) = \tilde\zeta_i\mid_{B^{+}\otimes B^{+}}(b_1\otimes b_2)$ for $b_1, b_2\in B^{+}.$ Indeed $\tilde\zeta_i$ is a coboundary on $B$ that is $\tilde\zeta_i = -\delta^{*} h_i$ where $h_i(r)$ is the coefficient of $x_i^{N_i}$ in $r\in B^+$ written as a linear combination of PBW basis elements. To see this note that $h_i : B\otimes B^{+}\rightarrow k$ is a 1-cochain, Hom$_B(B\otimes B^+,k)\cong$ Hom$_k(B^+,k)$ and $\delta^{*}h_i\in$ Hom$_B(B\otimes B^{+}\otimes B^{+},k)$.\\
\indent To define a 2-cocycle $\zeta_i$ on $A$ we next show that $\tilde\zeta_i$ factors through the quotient map $\pi : B\rightarrow A$ and that $\zeta_i$ is not a coboundary on $A$. We must show that $\tilde\zeta_i(r,s) = 0$ whenever either $r$ or $s \in$ Ker $\pi$. Consider the following diagram
 $$\xymatrix{B^{+}\otimes B^{+}\ar[d]_{\pi\otimes\pi} \ar[r]^{\mspace{45mu}\tilde{\zeta_i}} &k \\
                               A\otimes A\ar@{.>}[ur]_{\zeta_i}}$$
Suppose $\ x^a\in$ Ker $\pi\ $ then $a_j\geq N_j$ for some $j$ with $1\leq j\leq t.$ 
As per the assumption that $x_i^{N_i}$ is in the braided center we can write $x^a = \vartheta x_j^{N_j} x^b$ where $\vartheta$ is a non-zero scalar and $b$ is arbitrary. Therefore, $\tilde\zeta_i(x^a\otimes x^c)= \vartheta \tilde\zeta_i(x_j^{N_j}x^b\otimes x^c)$ and this is the coefficient of $x_i^{N_i}$ in the product $\vartheta x_j^{N_j}x^bx^c.$ If $j=i$, then since $x^c=x_1^{c_1}x_2^{c_2}\cdots x_n^{c_n} \in B^{+}$ the above product cannot have non-zero coefficient for $x_i^{N_i}$. The same is true, if $j\neq i$ since $x_j^{N_j}$ is a factor of $x^ax^c$. If $x^c \in$ Ker~$\pi$ a similar argument will work.\\
\indent Thus, we have $\tilde\zeta_i(x^a\otimes x^c) = 0$ that is, $\tilde\zeta_i$ factors through the quotient map $\pi : B\rightarrow A.$\\
Therefore, we may define $\zeta_i : A^{+}\otimes A^{+}\rightarrow k$ by
$$ \zeta_i(r\otimes s) = \tilde\zeta_i(\tilde r\otimes \tilde s)$$
where $\tilde r, \tilde s$ are defined via a section of $\pi$. (Choose the section $\phi$ of the quotient map $\pi : B\rightarrow A$ such that $\phi(r) = \tilde r$ where $\tilde r$ is the unique element that is a linear combination of the PBW basis elements of $B$ with $i_l< N_l \hspace{0.2cm} \text{ for all } l=1,\cdots,n$).\\
\indent This is well defined since $\tilde\zeta_i$ is well defined. We still need to verify that $\zeta_i$ is associative on $A^{+}$. Let $r, s, u\in A^{+}$ and since $\pi$ is algebra homomorphism $\tilde{r} \tilde {s} = \widetilde{rs} + y$ and $\tilde {s} \tilde{ u} = \widetilde{su} + z$ for some $y, z \in$ Ker $\pi.$ Observe that Ker $\pi\otimes B + B\otimes$ Ker $\pi \subset$ Ker~$\tilde\zeta_i.$\\
Therefore, we have
\begin{eqnarray*}
\zeta_i(rs\otimes u) &=& \tilde\zeta_i(\widetilde{rs}\otimes \tilde u)\\
&=&\tilde\zeta_i((\tilde {r} \tilde s-y)\otimes \tilde u)\\
&=&\tilde\zeta_i(\tilde {r} \tilde s\otimes \tilde u)\\ 
&=&\tilde\zeta_i(\tilde r\otimes \tilde {s} \tilde u)\hspace{.5cm} (\tilde\zeta_i \text{ associative })\\
&=&\tilde\zeta_i(\tilde r\otimes \widetilde{su})\\ 
&=&\zeta_i(r\otimes su)
\end{eqnarray*} 
This shows that $\zeta_i$ is associative on $A^{+}$. Hence, $\zeta_i$ is 2-cocycle on $A$.
\end{section}
\begin{section}{Finite Generation}
\indent In this section we prove our main theorem. We follow the same terminology as used in Section 5 of [16] with some additional information.\\\\
\indent Let $B$ be a PBW algebra as defined in Section 2 and $A=B/(x_1^{N_1},\cdots,x_t^{N_t})$. Recall the assumption from Section 4 that $x_i^{N_i}$ is in the braided center. Hence, a filtration on $B$ induces a filtration on $A$ [5, Theorem 4.6.5] for which $S = Gr A$, given by generators and relations of type (\ref{Ech3}). Thus H$^{*}(S,k)$ is given by Theorem \ref{Tch3}.\\
\indent Now our algebra $A$ is an augmented algebra over the field $k$, with augmentation $\varepsilon : A\rightarrow k$. Since $A$ is filtered it induces an increasing filtration $F_{0}P_{\bullet}\subset F_{1}P_{\bullet}\subset\cdots \subset F_{n}P_{\bullet}\subset\cdots$ on the reduced bar (free $A$) resolution of $k$,
$$P_\bullet : \cdots\stackrel{\partial_3}\rightarrow A\otimes (A^{+})^{\otimes 2} \stackrel{\partial_2}\rightarrow A\otimes A^{+} \stackrel{\partial_1}\rightarrow A \stackrel{\varepsilon}\rightarrow k\rightarrow 0$$
where $A^{+} =$ Ker $\varepsilon$, $\partial_{n}(a_0\otimes\cdots\otimes a_n) = \sum_{j=0}^{n-1} (-1)^j a_0\otimes\cdots\otimes a_ja_{j+1} \otimes\cdots\otimes a_n$ and the filtration is given in each degree $n$ by $$F_p(A\otimes (A^{+})^{\otimes n})=\displaystyle{\sum_{i_0+\cdots+i_n =p}} F_{i_0}A\otimes F_{i_1}(A^+)\otimes\cdots\otimes F_{i_n}(A^+).$$
Then the reduced bar complex of $Gr A$ is precisely $Gr P_\bullet$, where 
$$(Gr P_n)_p := F_{p}P_n/F_{p-1}P_n.$$ 
\indent Now let $\mathcal{C}^{\bullet}(A):=$ Hom$_{A}(P_{\bullet}, k)$. Note that $\mathcal{C}^{n}(A)=$ Hom$_{A}(P_{n}, k) =$ Hom$_{A}(A\otimes (A^{+})^{\otimes n}, k)$ is a filtered vector space where
$$F^{p}\mathcal{C}^{n}(A) = \{f: P_n\rightarrow k\mid f\mid_{F_{p-1}P_n} = 0\}$$
This filtration is compatible with the coboundary map on $\mathcal{C}^{\bullet}(A)$. Hence, $\mathcal{C}^{\bullet}(A)$ is a filtered cochain complex: $\mathcal{C} = F^{0}\mathcal{C}^{\bullet}\supset F^{1}\mathcal{C}^{\bullet}\supset\cdots .$ Now our algebra $A$ satisfies $F_{p}A=0$ if $p<0$, $1\in F_{0}A$ and $A=\bigcup_{p} F_{p}A$. Thus, there is a convergent May spectral sequence associated to the filtration of a cochain complex (see [18, Theorem 3] and [19, Theorem 12.5]):
\begin{equation}\label{E1ch5}
E_1^{p,q}= \text{H}^{p+q}((Gr A)_p ,k)\Longrightarrow \text{H}^{p+q}(A,k).
\end{equation}
{\bf Note:} For special cases refer to [22, Theorem 5.5.1].\\\\
From Section 4 we know that
\begin{equation}\label{E2ch5}
\zeta_i(x^a\otimes x^b) = \gamma_i
\end{equation}\\
where $\gamma_i$ is the coefficient of $x_i^{N_i}$ in the product $x^a x^b $, and $x^a,x^b$ range over all pairs of PBW basis elements. Recall that any PBW basis element is written as $x_1^{\alpha_1}\cdots x_n^{\alpha_n}$ and needs to be totally order. By [5, Theorem 4.6.5] we know that there exists a filtration and thus, there is a total ordering which we denote by $p_i$ (a positive integer). Now, observe that $\zeta_i$ are in degrees $(p_i,2-p_i)$.

We wanted to relate these functions $\zeta_i$ to the elements of the $E_1$ page of the spectral sequence (\ref{E1ch5}). We have $\zeta_i\mid_{F_{p_{i}-1}(A\otimes A)} = 0$ but $\zeta_i\mid_{F_{p_{i}}(A\otimes A)}\neq 0$ by (\ref{E2ch5}).\\
Thus, we conclude by the definition of $\zeta_i$ from Section 4 that $\zeta_i \in F^{p_i}\mathcal{C}^2$ but $\zeta_i \notin F^{p_i+1}\mathcal{C}^2$. The filtration on $\mathcal{C}^{\bullet}$ induces a filtration on H$^{*}(\mathcal{C}^{\bullet})$, that is to say $F^{p}$H$^{n}(\mathcal{C}^{\bullet}):=$ im$\{$H$^n(F^{p}\mathcal{C}^{\bullet})\rightarrow$ H$^n(\mathcal{C}^{\bullet})\}$ with $F^{0}$H$^n(\mathcal{C}^{\bullet})=\ $H$^n(\mathcal{C}^{\bullet})$. By denoting the corresponding cocycle in $F^{p_i}$H$^2(A,k)$ by the same letter we further conclude that $\zeta_i\in$ im$\{$H$^2(F^{p_i}\mathcal C^{\bullet})\rightarrow$ H$^2(\mathcal{C}^{\bullet})\} = F^{p_i}$H$^2(A,k)$, but $\zeta_i\notin$ im$\{$H$^2(F^{p_i+1}\mathcal{C}^{\bullet})\rightarrow$ H$^2(\mathcal{C}^{\bullet})\} = F^{p_i+1}$H$^2(A,k).$ Hence, we can identify $\zeta_i$ with corresponding nontrivial homogeneous element in the associated graded complex:
$$\tilde \zeta_i \in F^{p_i}\text{H}^2(A,k)/F^{p_i+1}\text{H}^2(A,k)\simeq E_{\infty}^{p_i,2-p_i}.$$
Refer to [18] for the isomorphism.\\
\indent Since $\zeta_i \in F^{p_i}\mathcal{C}^2$ but $\zeta_i \notin F^{p_i+1}\mathcal{C}^2$, it induces an element $\bar{\zeta_i}\in E_0^{p_i,2-p_i} = F^{p_i}\mathcal{C}^2/F^{p_i+1}\mathcal{C}^2$ which will be in the kernels of all the differentials of the spectral sequence since it is induced by an actual cocycle in $\mathcal{C}^{\bullet}$. Hence, the image of $\bar{\zeta_i}$ will be in the $E_{\infty}$-page. Now the non-zero element $\tilde \zeta_i$ is also induced by the same cocycle as $\bar{\zeta_i}$ in $\mathcal{C}^{\bullet}$. Hence we may identify these cocycles. This leads to the conclusion that $\tilde{\zeta_i}\in E_0^{p_i,2-p_i}$, and, correspondingly, its image in $E_1^{p_i,2-p_i}\hookrightarrow$ H$^2(Gr A,k)$ which we denote by the same symbol, is a permanent cycle.\\
\indent Note that via the formula (\ref{E2ch5}) we can obtain similar cocycles $\hat{\zeta_i}$ for $S = Gr A$. Comparing the values of $\bar{\zeta_i}$ and $\hat{\zeta_i}$ on basis elements $x^a\otimes x^b$ of $Gr A\otimes Gr A$ leads us to the conclusion that they are the same function. Hence $\hat{\zeta_i}\in E_1^{p_i,2-p_i}$ are permanent cycles.\\

We will identify these elements $\hat{\zeta_i}\in$ H$^2(Gr A,k)$ with the cohomology classes $\xi_i\in$ H$^{*}(S,k)$ of Theorem \ref{Tch3} via the following theorem.
\begin{theorem} For  each $i\ (1\leq i\leq n)$, the cohomology classes $\xi_i$ and $\hat{\zeta_i}$ coincide as elements of H$^2(Gr A,k)$.
\end{theorem}
\begin{proof}
In Section 3 we have defined the chain complex $K_{\bullet}$ which is a projective resolution of the trivial $Gr A$-module $k$. Elements $\eta_i\in$ H$^{1}(Gr A,k)$ and $\xi_i\in$ H$^{2}(Gr A,k)$ were defined via the complex $K_{\bullet}$. Our aim is to identify $\xi_i$ with the elements of the chain complex $\mathcal{C}^{\bullet}$ defined above. For this we consider the following diagram and define the maps $F_1, F_2$ making it commutative, where $S=Gr A$:
$$\xymatrix{\cdots\ar[r] & K_2\ar[d]^{F_2} \ar[r]^d & K_1\ar[d]^{F_1} \ar[r]^d & K_0\ar@{=}[d] \ar[r]^{\varepsilon} & k\ar@{=}[d] \ar[r] & 0\\
            \cdots\ar[r] & S\otimes(S^{+})^{\otimes 2}\ar[r]^{\partial_2}            & S\otimes(S^{+}) \ar[r]^{\partial_1} & S \ar[r]^{\varepsilon}       & k \ar[r]           & 0}
$$
where the map $d=d_1+d_2+\cdots d_n$ is defined in Section 3 and $\partial_i(s_0\otimes s_1\otimes\cdots\otimes s_i) = \sum_{j=0}^{i-1}(-1)^js_0\otimes\cdots\otimes s_js_{j+1}\otimes\cdots\otimes s_i$ is defined in Section 4. Let $\Phi(\cdots 1_i\cdots)$ where 1 is in the $i$th position and 0 in all other positions denote the basis element of $K_1$, $\Phi(\cdots 1_i\cdots 1_j\cdots)$ (respectively $\Phi(\cdots 2_i\cdots) \text{ for } i\leq t$) where 1 is in the $i$th and $j$th positions $(i\neq j)$, and 0 in all other positions (respectively a 2 in the $i$th position and 0 in all other positions) denote the basis element of $K_2$. Let
\begin{align*}
F_1(\Phi(\cdots 1_i\cdots))&= 1\otimes x_i,\\
F_2(\Phi(\cdots 2_i\cdots))&= \sum_{a_i = 0}^{N_i-2}x_i^{a_i}\otimes x_i\otimes x_i^{N_{i}-a_{i}-1},\\
F_2(\Phi(\cdots 1_i\cdots 1_j\cdots))&= 1\otimes x_j\otimes x_i - q_{ji}\otimes x_i\otimes x_j
\end{align*}
\indent We want to provide a chain map $F_{\bullet} : K_{\bullet}\rightarrow S\otimes (S^+)^{\otimes \bullet}$ by extending $F_1, F_2$ to maps $F_i : K_i\rightarrow S\otimes (S^+)^{\otimes i}, i\geq 1$. This can be done by showing that the two nontrivial squares in the above diagram commute.\\
\indent Consider, 
\begin{align*}
d(\Phi(\cdots 1_i\cdots))&= (d_1+\cdots + d_i+\cdots + d_n)(\Phi(\cdots 1_i\cdots))\\
&=x_i\Phi(\cdots 0_i\cdots)\\
&=x_i\\\\
\partial_1\circ F_1(\Phi(\cdots 1_i\cdots))&=\partial_i(1\otimes x_i)\\
&=1\cdot x_i\\
&=x_i
\end{align*}
Thus, we have $d=\partial_1\circ F_1$. Similarly, we can check that $F_1\circ d = \partial_2\circ F_2$.\\
\indent Hence, two nontrivial squares in the above diagram commute. So by the Comparison Theorem [14]  there exists a chain map $F_{\bullet} : K_{\bullet}\rightarrow S\otimes (S^+)^{\otimes \bullet}$ that induces an isomorphism on cohomology.\\
\indent We now verify that the maps $F_1, F_2$ give the desired identifications. Here we use the definition in (\ref{E2ch5}) to represent the function $\xi_i$ on the reduced bar complex, $\xi_i(1\otimes x^a\otimes x^b):= \xi_i(x^a\otimes x^b)$. Then
\begin{eqnarray*}
&&\hspace{-1.72in}F_2^{*}(\xi_i)(\Phi(\cdots 2_i\cdots))=\xi_i(F_2(\Phi(\cdots 2_i\cdots)))\\
&=&\xi_i(\sum_{a_i = 0}^{N_i-2}x_i^{a_i}\otimes x_i\otimes x_i^{N_{i}-a_{i}-1})\\
&=&\sum_{a_i = 0}^{N_i-2}\varepsilon(x_i^{a_i})\xi_i(1\otimes x_i\otimes x_i^{N_{i}-a_{i}-1})\\
&=&\xi_i(x_i\otimes x_i^{N_{i}-1})\\
&=&1
\end{eqnarray*}
Similarly, we can check that $F_2^{*}(\xi_i)(\Phi(\cdots 1_i\cdots 1_j\cdots)) = 0$ for all $i, j$ and $F_2^{*}(\xi_i)(\Phi(\cdots 2_j\cdots)) = 0$ for all $j\neq i$.\\
Therefore, $F_2^{*}(\xi_i)$ is the dual function to $\Phi(\cdots 2_i\cdots)$ which is precisely $\xi_i$.\\
\end{proof}
\indent In the same manner, we identify the elements $\eta_i$ defined above with functions at the chain level in cohomology. For that define
$$\eta_i(x^a) =
\begin{cases}
1, \text{ if } x^a=x_i\\
0, \text{ otherwise }\\
\end{cases}
$$
The functions $\eta_i$ represent a basis of H$^1(S,k)\simeq$ Hom$_k(S^{+}/(S^+)^2,k)$. Consider,
\begin{eqnarray*}
&&\hspace{-1.75in}F_1^{*}(\eta_i)(\Phi(\cdots 1_j\cdots))=\eta_i(F_1(\Phi(\cdots 1_j\cdots)))\\
&=&\eta_i(1\otimes x_j)\\
&=&\eta_i(x_j)\\
&=&\begin{cases}
1,  \text{ if } j=i\\
0,  \text{ otherwise }
\end{cases}\\
&=&\delta_{ij}
\end{eqnarray*}
Thus $F_1^{*}(\eta_i)$ is the dual function to $\Phi(\cdots 1_i\cdots)$. Therefore $\eta_i$ and $\hat{\eta_i}$ coincide as elements of H$^1(S,k)$ where $\hat{\eta_i}$ is a 1-cocycle of $A$.\\
\begin{theorem}
The cohomology algebra {\em H}$^{*}(A,k)$ is finitely generated.
\end{theorem}
\begin{proof}
Let $E_1^{*,*} \Longrightarrow$ H$^{*}(A,k)$ be the May spectral sequence and $D^{*,*}$ be the bigraded subalgebra of $E_1^{*,*}$ generated by the elements $\xi_i$. So by the above discussion $D^{*,*}$ consists of permanent cycles and $\xi_i$ is in bidegree $(p_i,2-p_i)$. Moreover, $D^{*,*}$ is Noetherian since it is a quantum polynomial algebra in $\xi_i$ [11]. By Theorem \ref{Tch3} the algebra $E_1^{*,*}$ is generated by $\xi_i$ and $\eta_i$ where the generators $\eta_i$ are nilpotent. Since $D^{*,*}$ is a subalgebra of $E_1^{*,*}$, we get an inclusion map $f: D^{*,*}\rightarrow E_1^{*,*}$ making $E_1^{*,*}$ a module over $D^{*,*}$. Hence, $E_1^{*,*}$ is a finitely generated module over $D^{*,*}$ and is generated by $\eta_1,\cdots,\eta_n$. Therefore, by Lemma \ref{Lch2},  $E_\infty^{*}$ is a Noetherian Tot($D^{*,*}$)-module. But $E_\infty^{*}\cong Gr$ H$^{*}(A,k)$ [18]. Thus, $Gr$ H$^{*}(A,k)$ is a Noetherian Tot($D^{*,*}$)-module and hence is finitely generated. Therefore, H$^{*}(A,k)$ is finitely generated.\\
\end{proof}
Thus, this leads us to the question whether H$^{*}(A,M)$ is a finitely generated module over H$^{*}(A,k)$ where $M$ is a finitely generated $A$-module? This is true in special cases for e.g. 1) $A$ is a finite dimensional Hopf algebra [16], 2) $A$ is restricted enveloping algebra of restricted Lie superalgebras [1], 3) $A$ is a Frobenius-Lusztig kernel [6] and 4) $A$ is restricted enveloping algebra of classical Lie superalgebras [15].\\\\
\end{section}
\noindent{\bf Acknowledgement:} This paper is based on author's PhD thesis. The author would like to thank his advisor Prof. Sarah Witherspoon for all her support, encouragement and guidance.\\

\end{document}